\documentclass[10pt,twoside]{homg3} 
\usepackage{amssymb}
\usepackage{amsmath}

\newcommand{\proof}{{\em Proof.}\quad}

\renewcommand{\remark}{{\noindent\em Remark}\quad}

\def\TituloCorto{Fixed points, bounded orbits and attractors}
\def\AutorCorto{H. Barge and J.M.R. Sanjurjo}

\begin{document}

\tittle{Fixed points, bounded orbits and attractors of planar flows}

\def\authors{\aaa{H\'ector  B{\surname ARGE}
\hspace{-1mm}
\ and  Jos\'e M.R.  S{\surname ANJURJO}
\hspace{-1mm}\thanks{The authors are partially supported by project MTM2012-30719.}
}}

\def\direc{\address{
Departamento de Geometr\'ia y Topolog\'ia\\
Universidad Complutense de Madrid \\
28040 Madrid, Spain \\
hbarge@ucm.es \\ \vskip .2cm
Departmento de Geometr\'ia y Topolog\'ia \\
Universidad Complutense of Madrid \\
28040 Madrid, Spain\\
jose\_sanjurjo@mat.ucm.es}}

\maketitle

\begin{center}
{\em Affectionately dedicated to Jos\'e Mar\'ia Montesinos-Amilibia.}
\end{center}

\vspace{0.1cm}

\begin{abstract}
\noindent In this paper we provide a dynamical characterization of isolated invariant continua which are global attractors for planar dissipative flows. As a consequence, a sufficient condition for an isolated invariant continuum to be either an attractor or a repeller is derived for general planar flows. 
\end{abstract}


\par
\bigskip

\section{Introduction}

In this paper we are concerned with the study of planar flows $\varphi:\mathbb{R}^{2}\rightarrow \mathbb{R}^{2}.$ In particular, we provide a
dynamical characterization of isolated invariant continua which are global attractors for planar dissipative flows. This characterization is inspired
by a result of Alarc\'{o}n, Gu\'{\i}\~{n}ez and Guti\'{e}rrez about dissipative planar embeddings with only one fixed point (see \cite{Alarcon}). Moreover we will derive a sufficient condition for a planar continuum to be an attractor or a repeller provided that it contains all the fixed points of $\varphi$.

We shall use through the paper the standard notation and terminology in the theory of dynamical systems. In particular we shall use the notation $\gamma(x)$ for the \emph{trajectory} of the point $x$, i.e. $\gamma(x)=\{xt\mid t\in\mathbb{R}\}$. By the \emph{omega-limit} of a point $x$ we understand the set $\omega(x)=\bigcap_{t>0}\overline{x[t,\infty)}$ while the \emph{negative omega-limit} is the set $\omega^*(x)=\bigcap_{t<0}\overline{x(-\infty,t]}$. An invariant compactum $K$ is \emph{stable} if every neighborhood $U$ of $K$ contains a neighborhood $V$ of $K$ such that $V[0,\infty)\subset U$. Similarly, $K$ is \emph{negatively stable} if every neighborhood $U$ of $K$ contains a neighborhood $V$ of $K$ such that $V(-\infty,0]\subset U$. An invariant compactum $K$ is said to be \emph{attracting} provided that there exists a neighborhood $U$ of $K$ such that $\omega(x)\subset K$ for every $x\in U$. In an analogous way, $K$ is said to be \emph{repelling} provided that there exists a neighborhood $U$ of $K$ such that $\omega^*(x)\subset K$ for every $x\in U$. An \emph{attractor} (or \emph{asymptotically stable} compactum) is an attracting stable set and a \emph{repeller} is a repelling negatively stable set. 

If $K$ is an attracting set, its region (or basin) of attraction $\mathcal{A}$ is the set of all points $x\in M$ such that $\omega(x)\subset K$. An attracting set $K$ is \emph{globally attracting} provided that $\mathcal{A}$ is the whole phase space. If $K$ is an attractor and $\mathcal{A}$ is the whole phase space, then $K$ is said to be a \emph{global attractor} (or \emph{globally asymptotically stable} compactum). For the reader interested in a detailed treatment of attracting sets we recommend \cite{Moron} and \cite{Sanchez-Gabites Transactions}.

Through this paper we shall deal with a special kind of invariant compacta, the so-called \emph{isolated invariant sets} (see \cite{Conley,Conley-Easton,Easton,Salamon} for details). These are compact invariant sets $K$ which possess an \emph{isolating neighborhood}, i.e. a compact neighborhood $N$  such that $K$ is the maximal invariant set in $N$. For instance, attractors and repellers are isolated invariant sets. We shall make use of the next result which states that isolated globally attracting continua for planar flows are stable.

\begin{theorem}[Mor\'on, S\'anchez-Gabites and Sanjurjo  \cite{Moron}\label{thm1}]
Every connected \\ isolated globally attracting set $K$ in $\mathbb{R}^2$ is a global attractor.
\end{theorem}

 A special kind of isolating neighborhoods shall be useful in the sequel, the so-called \emph{isolating blocks}, which have good topological properties. More precisely, an isolating block $N$ is an isolating neighborhood such that there are compact sets $N^i,N^o\subset\partial N$, called the entrance and the exit sets, satisfying
\begin{enumerate}
\item $\partial N=N^i\cup N^o$;
\item for each $x\in N^i$ there exists $\varepsilon>0$ such that $x[-\varepsilon,0)\subset M-N$ and for each $x\in N^o$ there exists $\delta>0$ such that $x(0,\delta]\subset M-N$;
\item for each $x\in\partial N-N^i$ there exists $\varepsilon>0$ such that $x[-\varepsilon,0)\subset \mathring{N}$ and for every $x\in\partial N-N^o$ there exists $\delta>0$ such that $x(0,\delta]\subset\mathring{N}$.
\end{enumerate} 

These blocks form a neighborhood basis of $K$ in $M$. If the flow is {diffe\-ren\-tia\-ble}, the isolating blocks can be chosen to be differentiable manifolds which contain $N^i$ and $N^o$ as submanifolds of their boundaries and such that $\partial N^i=\partial N^o=N^i\cap N^o$. In particular, for flows defined on $\mathbb{R}^2$, the exit set $N^o$ is a disjoint union of a finite number of intervals $J_1,\ldots,J_m $ and circumpherences $C_1,\ldots,C_n$ and the same is true for the entrance set $N^i$.

The dynamical structure near isolating invariant sets shall play an important role in this paper and it is described by the

\begin{theorem}[Ura-Kimura-Egawa \cite{Ura-Kimura, Egawa}]
Let $M$ be a locally compact separable metric space and $\varphi$ a flow on $M$. Suppose $K\neq M$ is a non-empty isolated invariant compactum. Then, one and only one of the  following alternatives holds:
\begin{enumerate}
\item $K$ is an attractor;
\item $K$ is a repeller;
\item There exist points $x\in M-K$ and $y\in M-K$ such that $\emptyset\neq\omega(x)\subset K$ and $\emptyset\neq\omega^*(y)\subset K$. 
\end{enumerate}
\end{theorem}

We shall also make use of a classical result of C. Guti\'{e}rrez about smoothing of $2$-dimensional flows.

\begin{theorem}[Guti\'{e}rrez \cite{Gutierrez}] Let $\varphi :M\times \mathbb{R}\rightarrow M$ be a continuous flow on a compact $C^{\infty }$ two-manifold $M$. Then there
exists a $C^{1}$ flow $\psi $ on $M$ which is topologically equivalent to $\varphi $. Furthermore, the following conditions are equivalent:
\begin{enumerate}
\item any minimal set of $\varphi $ is trivial;

\item $\varphi $ is topologically equivalent to a $C^{2}$ flow;

\item $\varphi $ is topologically equivalent to a $C^{\infty }$ flow.
\end{enumerate}
\end{theorem}

By a trivial minimal set we understand a fixed point, a closed trajectory or the whole manifold if $M$ is the $2$-dimensional torus and $\varphi$ is topologically equivalent to an irrational flow. We readily deduce from Guti\'{e}rrez' Theorem applied to the Alexandrov compactification of the plane that continuous flows $\varphi :\mathbb{R}^{2}\times \mathbb{R}\rightarrow \mathbb{R}^{2}$ are topologically equivalent to $C^{\infty }$ flows.

Some basic results about planar vector fields such as, the Poincar\'e-Bendixson Theorem, the Tubular Flow Theorem and the elementary properties of transversal sections shall be used through the paper. Two good {re\-fe\-ren\-ces} covering this material are the book of Hirsch, Smale and Devaney \cite{Hirsch} and the monograph of Palis and de Melo \cite{Palis}. In addition, a form of homotopy theory, namely \emph{shape theory}, which is the most {sui\-ta\-ble} for the study of global topological properties in dynamics, will be {occa\-sio\-na\-lly used}. Although shape theory is not necessary to understand the paper, we {re\-co\-mmend} to the reader the references \cite{Borsuk}, which contains an exhaustive treatment of the subject and \cite{Sanjurjo}, which covers some dynamical applications of this theory. 
 
\section{Planar dissipative systems and isolated {in\-va\-riant} continua}

We start this section by recalling the definition of \emph{dissipative flow}. Let $M$ be a locally compact metric space and $\varphi:M\times\mathbb{R}\to M$ a flow on $M$. The flow $\varphi$ is said to be \emph{dissipative} if $\omega(x)\neq\emptyset$  for every $x\in M$ and $\bigcup_{x\in M}\omega(x)$ has compact closure. If the phase space $M$ is not compact, dissipativeness is equivalent to $\{\infty\}$ being a repeller of the extended flow $\widehat{\varphi}:(M\cup\{\infty\})\times\mathbb{R}\to M\cup\{\infty\}$ to the Alexandrov compactification of $M$ leaving $\infty$ fixed (See\cite{Garay,Hale,Sanjurjo}), and therefore to the existence of a global attractor for $\varphi$.

 The following result gives a relation between global asymptotic stability of a fixed point and the non-existence of {addi\-tio\-nal} fixed points in the case of discrete dynamical systems.   

\begin{theorem}[Alarc\'on-Gu\'i\~nez-Guti\'errez \cite{Alarcon}, Ortega-Ruiz del Portal \cite{Ortega 2}\label{thm2}]
Assume that $h\in \mathcal{H}_+$ (orientation preserving homeomorphisms of $\mathbb{R}^2$) is dissipative and $p$ is an asymptotically stable fixed point of $h$. The following conditions are {equi\-va\-lent}:
\begin{enumerate}
\item $p$ is globally asymptotically stable;

\item $Fix(h)=p$ and there exists an arc $\gamma\subset S^2$ with end points at $p$ and $\infty$ such that $h(\gamma)=\gamma$.
\end{enumerate}
\end{theorem}

The proof in \cite{Alarcon} is based on Brouwer's theory of fixed point free homeomorphisms of the plane. Ortega and Ruiz del Portal give in \cite{Ortega 2} an alternative proof based on the theory of prime ends.

Inspired by Theorem~\ref{thm2}, the authors prove in \cite{Barge} that for continuous and dissipative dynamical systems the result is satisfied even if the fixed point $p$ is substituted by a connected attractor $K$ which contains every fixed point of the flow. We prove in our next result that the asymptotical stability condition can be dropped from the hypothesis, obtaining in this way a simple characterization of global attractors of dissipative planar flows.

\begin{theorem}\label{thm3}
Let $K$ be an isolated invariant continuum of a dissipative flow $\varphi$ in $\mathbb{R}^2$. The following conditions are equivalent:
\begin{enumerate}
\item $K$ is a global attractor;

\item There are no fixed points in $\mathbb{R}^2-K$ and there exists an orbit $\gamma$ connecting $\infty$ and $K$ (i.e. such that $||\gamma(t)||\to \infty$ when $t\to -\infty$ and $\omega(\gamma)\subset K)$.
\end{enumerate}
\end{theorem}
\proof By the Guti\'errez Theorem \cite{Gutierrez} we can assume that the flow $\varphi $ is differentiable. Since $\varphi $ is dissipative, given $x\in \mathbb{R}^{2}$ its $\omega $-limit is non-empty and compact. Moreover, by the Poincar\'{e}-Bendixson Theorem either $\omega (x)$ contains fixed points and, hence, $\omega (x)\cap K\neq \emptyset $ or $\omega (x)$ is a periodic orbit. If $\omega (x)$ is a periodic orbit then $K$ is not contained in its interior
since, in that case, $\gamma$ would meet $\omega (x)$, which is impossible. Therefore, if $\omega (x)$ is not contained in $K$, then $K$ is in the
exterior of $\omega (x)$ and, moreover, $\omega (x)$ being a periodic orbit, there must exist a fixed point in its interior. Hence this point belongs to $
K$, which is a contradiction$.$ We conclude that if $\omega (x)$ is a periodic orbit then $\omega (x)\subset K.$

If $\omega (x)$ is not a periodic orbit then $\omega (x)\cap K\neq \emptyset$ and we shall prove that, in fact, $\omega (x)\subset K$. We suppose, to
get a contradiction, that there exists $y\in \omega (x)-K$. By hypothesis $y$ is not a fixed point and, thus, we can take a local section $I$ containing $
y $ and meeting transversally the trajectory of $y$. Since $y\notin K$ we can assume that $I\cap K=\emptyset $. It is a well-known fact that the
trajectory of $x$ meets $I$ infinitely many times. We consider two consecutive points of intersection $x_{1}=xt_{1}$ and $x_{2}=xt_{2}$ with $
x_{1}$,$x_{2}\in I$, $0<t_{1}<t_{2}$ and $x[t_{1},t_{2}]\cap I=\{x_{1}$,$x_{2}\}.$ Then the set $C=x[t_{1},t_{2}]\cup J$, where $J$ is the
subinterval of $I$ bounded by $x_{1}$ and $x_{2}$, is a simple closed curve which, by the Jordan Theorem, decomposes $\mathbb{R}^{2}$ into two connected
components $U$ and $V$. If $U$ is the bounded component then $U$ is either positively or negatively invariant by \cite{Hirsch}. Then, a simple argument involving again the Poincar\'{e}-Bendixson Theorem, leads to the existence of a fixed point in $U$ which, by hypothesis, belongs to $K$. Now, the intersection of $K$ with $C$ is empty, which implies that $K\subset U\cup V$ and, $K$ being connected, that $K\subset U$. If $U$ is negatively invariant, the trajectory $\gamma $ linking $\infty $ with $K$ cannot \ enter in $U$ since the only possibility would be through $J$, which is an exit set. This makes it impossible that $\omega(\gamma )\subset K$ and we get a contradiction with the hypothesis. If $U$ is positively invariant then an easy argument shows that $y\in\omega(\gamma)$ in contradiction with the assumption. This proves that $\omega (x)\subset K$ for every $x\in \mathbb{R}^{2}$ and, as a consequence, $K$ is a globally attracting set. Since $K$ is isolated, by Theorem~\ref{thm1} $K$ must be stable, i.e. a global attractor. This establishes the implication $2.\Rightarrow 1.$ The converse implication is straightforward.

\section{Attractors, repellers and bounded orbits}

We present in this section a result which gives a readily testable sufficient condition for a planar compactum to be either an attractor or a repeller.

\begin{theorem}\label{thm4}
Let $K$ be an isolated invariant continuum of a flow $\varphi $ in $\mathbb{R}^{2}$. Suppose that there is a closed disk $D$ containing $K$ in its
interior such that there are no fixed points in $D-K$ and that there is an orbit $\gamma $ completely contained in $D-K$. Then $K$ is either an attractor or a repeller. Moreover, $K$ has trivial shape.
\end{theorem}
\proof We can assume again that $\varphi $ is differentiable. Since $\overline{\gamma}\subset D$ we have that $\omega (\gamma )\subset D$ and $\omega ^{\ast
}(\gamma )\subset D$. We start by proving that there exists an orbit $\Gamma$ in $D-K$ satisfying the additional condition that either $\omega (\Gamma )\subset K$ or $\omega ^*(\Gamma )\subset K.$ As a consequence of the Poincar\'{e}-Bendixson Theorem and the hypothesis of the present theorem we have
that either $\omega (\gamma )\cap K\neq \emptyset $ or $\omega (\gamma )$ is a periodic orbit not meeting $K,$ and the same can be said for $\omega
^*(\gamma ).$ If $\omega (\gamma )$ is a periodic orbit not meeting $K$ then $K$ is in its interior and, by the Ura-Kimura Theorem, there exists a
point $x,$ also in the interior of $\omega (\gamma )$, with $\omega(x)\subset K$ or $\omega ^*(x)\subset K$, and the same happens if $\omega ^*(\gamma )$ is a periodic orbit not meeting $K.$ Hence, in both cases $\Gamma $ can be taken as the trajectory of $x$. On the other
hand, we will prove that the possibility that both intersections, $\omega(\gamma )\cap K$ and $\omega ^*(\gamma )\cap K,$ are non-empty can
never happen. Suppose, to get a contradiction, that $\omega (\gamma )\cap K\neq \emptyset $ and $\omega ^*(\gamma )\cap K\neq \emptyset .$ Take
an isolating block $N$ of $K$. By \cite{Conley-Easton} $N$ can be chosen to be a topological closed disk with $i$ holes, one for every bounded component of $\mathbb{R}^{2}-K$. We suppose that $\gamma$ is in the unbounded component $U$ (the argument being only slightly different in the other case) and consider the only circle $C\subset \partial N$ \ which is contained in $U$. Then, there exists a point $x\in C\cap \gamma $ leaving $N$ and returning to $N$ after a
time $t\neq 0$, i.e. such that $xt\in C$ and $x(0,t)\cap N=\emptyset .$ The possibility that the time $t$ be positive or negative is irrelevant in this
construction. Consider the arc $A$ in $C$ with extremes $x$ and $xt$ such that the topological circle $x[0,t]\cup A$ does not contain $K$ in its interior. This arc can be mapped to the unit interval $I=[0,1]$ of the real line by a homeomorphism $h:A\rightarrow I$. If we take the point $x_{1}\in \mathring{A}$ corresponding to the center of $I$ then $x_{1}$ must leave $N$ (in the past or in the future) and return again since, otherwise, the Theorem of Poincar\'{e}-Bendixson would imply the existence of a fixed point in the disk limited by $x[0,t]\cup A$. Hence, we can repeat the operation with $x_{1}[0,t_{1}]\cup A_{1}$, where $A_{1}$ is an arc in $A$ with extremes $x_{1}$ and $x_{1}t_{1}$ and the topological circle $%
x_{1}[0,t_{1}]\cup A_{1}$ does not contain $K$ in its interior. Now take $x_{2}\in $ $\mathring{A}_{1}$ corresponding to the middle point of $h(A_{1})$
and repeat the construction. In this way we obtain a sequence $A\supset A_{1}\supset A_{2}\supset\ldots$ of arcs whose intersection $\bigcap_{i=1}^\infty A_{i}$ consists of one point $p\in \partial N$. The orbit of $p$ defines an internal tangency to $\partial N,$ which is in contradiction with the properties of isolating blocks. We get from this contradiction that either $\omega (\gamma )\cap K=\emptyset $ or $\omega ^*(\gamma )\cap
K=\emptyset $ and, as a consequence, one of the two limits is a periodic orbit. Therefore, it follows from the remarks at the beginning of the proof
that there exists an orbit $\Gamma $ in $D-K$ satisfying the additional condition that either $\omega (\Gamma )\subset K$ or $\omega ^*(\Gamma
)\subset K.$

Suppose that $\omega (\Gamma )\subset K$. Then, $\omega ^*(\Gamma )$ is a periodic orbit containing $K$ in its interior. Let $V$ be the interior of $
\omega ^*(\Gamma )$ and consider the flow restricted to $\overline{V}$. An elementary argument involving local sections again shows that $\omega ^*(\Gamma )$ is a repeller for $\varphi |\overline{V}$ and, as a consequence, the restriction of $\varphi$ to $V$ is a dissipative flow. Then, using an
arbitrary homeomorphism between $V$ and $\mathbb{R}^{2}$ we can define a dissipative flow in $\mathbb{R}^{2}$ conjugated to $\varphi |V$  and
satisfying the conditions of Theorem~\ref{thm3}. We deduce from that theorem that $K$ is an attractor of $\varphi $ whose basin of attraction, $V$, is an open topological disk. Hence, $K$ has trivial shape by \cite{Kapitanski}. In the dual situation (when $\omega ^*(\Gamma )\subset K$ and $\omega (\Gamma )$ is a
periodic orbit containing $K$ in its interior), which could be discussed analogously using the reverse flow, it follows that $K$ is a repeller with trivial shape.

\vspace{3mm}

From Theorem~\ref{thm4} it follows:

\begin{corollary}\label{coro4}
Let $K$ be an isolated invariant continuum of a flow $\varphi $ in $\mathbb{R}^{2}$. Suppose that $K$ contains all the fixed points of $\varphi $ and
that there exists a bounded orbit $\gamma $ in $\mathbb{R}^{2}-K$. Then $K$ is either an attractor or a repeller. Moreover, $K$ has trivial shape.
\end{corollary}

\proof
The set $K\cup\overline{\gamma}$ is compact and as a consequence there exists a closed disk $D$ such that $K\cup\overline{\gamma}\subset \mathring{D}$. Then, Theorem~\ref{thm4} applies since the bounded orbit $\gamma\subset D-K$, and $D-K$ does not contain fixed points by assumption. 

\vspace{3mm}

\remark
The assumptions about the existence of a disk $D$ such that there is an entire orbit contained in $D-K$ in Theorem~\ref{thm4} and the existence of a bounded orbit in $\mathbb{R}^2-K$ in Corollary~\ref{coro4} are unavoidable. For instance, consider the flow $\varphi$ induced by the linear system
\begin{equation*}
\begin{cases}
\dot{x}=x\\
\dot{y}=-y
\end{cases}
\end{equation*}
The origin $(0,0)$ is a fixed point which is isolated as an invariant set and there are neither fixed points nor other bounded orbits in $\mathbb{R}^2-\{(0,0)\}$. In this case, $\{(0,0)\}$ is a saddle and hence, it is neither an attractor nor a repeller.

As a consequence of Corollary~\ref{coro4} and \cite[Theorem~12]{Barge} we obtain the following dichotomy for
dissipative flows:

\begin{corollary}
Let $K$ be an isolated invariant continuum of a dissipative flow $\varphi $ in $\mathbb{R}^{2}$. Suppose that $K$ contains all the fixed points of $\varphi$, then $K$ has trivial shape and is either an attractor or a repeller. Moreover, if $K$ is a repeller then there exists an attractor $K^*\subset\mathbb{R}^2-K$  which is either a limit cycle or homeomorphic to a closed annulus bounded by two limit cycles.
\end{corollary}

\begin{proof}
The dissipativeness of $\varphi$ guarantees the existence of a global attractor $K'$ and as a consequence $K\subset K'$. Suppose $K'\neq K$, since otherwise we have nothing to prove. Let $x\in K'-K$, the orbit $\gamma(x)$ is a bounded orbit being contained in the invariant compactum $K'$. Then, Corollary~\ref{coro4} ensures that $K$ is either an attractor or a repeller. This proves the first part of the statement.  

Suppose that $K$ is a repeller and consider the flow $\varphi|K'$, i.e. the restriction of $\varphi$ to the global attractor. The continuum $K$ is also a repeller for $\varphi|K'$ and then there exists an invariant compactum $K^*\subset K'$ such that the pair $(K^*,K)$ is an attractor-repeller decomposition of  $\varphi|K'$. Besides, the invariant compactum $K^*$ is an attractor for $\varphi$ since $K^*$ is an attractor for $\varphi|K'$ and $K'$ is an attractor. The region of attraction of $K^*$ agrees with $\mathbb{R}^2-K$ since $K$ is a repeller and $(K^*,K)$ is an attractor-repeller decomposition of the restriction of $\varphi$ to the global attractor $K'$. Moreover, $\mathbb{R}^2-K$ is connected $K$ being of trivial shape \cite{Borsuk} and hence so is $K^*$ by \cite{Kapitanski} and \cite{Borsuk}. We have proved that $K^*$ is a connected attractor which does not contain fixed points, thus by \cite[Theorem~12]{Barge} it must be either a limit cycle or homeomorphic to a closed annulus bounded by two limit cycles. 
\end{proof}

\endDocument 